\documentclass[10pt,a4paper]{article}
\usepackage{amsfonts}
\usepackage[latin9]{inputenc}
\usepackage{amsmath}
\usepackage{amssymb}
\usepackage{breqn}
\usepackage{url}
\usepackage{graphicx}
\usepackage[pdfstartview=FitH]{hyperref}
\usepackage{amsthm}
\usepackage{authblk}
\usepackage{hyperref}
\usepackage{url}
\usepackage{appendix}
\makeatletter
\begin{document}
\newtheorem{theorem}{Theorem}[section]
\newtheorem{lemma}[theorem]{Lemma}
\newtheorem{definition}[theorem]{Definition}
\newtheorem{claim}[theorem]{Claim}
\newtheorem{example}[theorem]{Example}
\newtheorem{remark}[theorem]{Remark}
\newtheorem{proposition}[theorem]{Proposition}
\newtheorem{corollary}[theorem]{Corollary}
\newtheorem{observation}[theorem]{Observation}

\title{High dimensional analogue of metric distortion for simplicial complexes}
\author{Izhar Oppenheim}
\affil{Department of Mathematics\\
 The Ohio State University  \\
 Columbus, OH 43210, USA \\
E-mail: izharo@gmail.com}

\maketitle
\textbf{Abstract}. We suggest a new possible high dimensional analogue to metric distortion. We then show a possible method for providing lower bounds to this distortion and use this method to prove a "Bourgain-type" distortion theorem for Linial-Meshulam random complexes.   

\section{Introduction}
In \cite{Bour}, Bourgain proved the following theorem:
\begin{theorem} \cite{Bour}[proof of Proposition 2]
Let $(V,E) = G \sim G(N,p)$, where $G(N,p)$ is the Erd\H{o}s-R\'enyi random graph on $N$ vertices. There are constants $C,K$ such that if $p = C \frac{\ln (N)}{N}$, then with high probability (i.e., the probability approaches $1$ as $N$ goes to $\infty$), for every map $g: V \rightarrow H$, where $H$ is a Hilbert space, the following holds:
$$\left( \max_{u,v \in V} \dfrac{\Vert g (u) - g(v) \Vert}{d(u,v)} \right) \left( \max_{u,v \in V} \dfrac{d(u,v)}{\Vert g (u) - g(v) \Vert} \right) \geq K \dfrac{\ln (N)}{\ln (\ln (N))}, $$
where $\Vert . \Vert$ denotes the norm in the Hilbert space and $d$ the graph metric in $G$. 
\end{theorem}

The left hand side of the inequality in the above theorem is sometimes called the metric distortion of $g$. It is natural to ask is there a high dimensional analogue of the above theorem regarding high dimensional distortion of simplicial complexes. To answer this question, one must first define what is high dimensional distortion. In \cite{Dot}, Dotterrer suggested a possible definition of high dimensional distortion in simplicial complexes. In this paper, we suggest a different definition, which is inspired by the work of Dotterrer, yet very different in nature. We use this definition to prove an analogue of Bourgain's theorem which is stated below, after some necessary definitions.
\begin{definition}
Fix $k \geq 0$. Let $X$ be a simplicial complex complete $k$-skeleton (i.e., any $k+1$ vertices of $X$ form a $k$-simplex). 
\begin{enumerate}
\item A $(k+1)$-gallery in $X$ is a sequence of $(k+1)$-dimensional simplices in $X$, $\sigma_1,...,\sigma_l$ such that for every $1 \leq i \leq l-1$, $\vert \sigma_i \cap \sigma_{i+1} \vert = k+1$.
\item Given $\eta_0 = \lbrace u_0,..., u_k \rbrace, eta_1 = \lbrace v_0,..., v_k \rbrace$, we shall say that a $(k+1)$-gallery $\sigma_1,...,\sigma_l$ connects $\eta_0$ and $\eta_1$ if  $\eta_0 \subset \sigma_0, \eta_1 \subset \sigma_l$, i.e., $\eta_0$ is a face of $\sigma_0$ and $\eta_1$ is a face of $\sigma_l$. $X$ will be called $(k+1)$-gallery connected, if for every two simplices $\lbrace u_0,..., u_k \rbrace, \lbrace v_0,..., v_k \rbrace$ are connected by a $(k+1)$-gallery.
\item For any set of vertices  $S = \lbrace u_0,...,u_{k+1} \rbrace$, we will say that a set $F$ of $(k+1)$-simplices is $(k+1)$-gallery filling of $S$ if every subsets 
$$\lbrace u_0,...,\widehat{u_i},...,u_{k+1} \rbrace, \lbrace u_0,...,\widehat{u_j},...,u_{k+1} \rbrace$$ are connected by a $(k+1)$-gallery of simplices of $F$. Define the $(k+1)$-gallery filling number of $S$ as
$$Fill_{k+1} (S) = \min \lbrace \vert F \vert : F \subseteq X^{(k+1)}, F \text{ is (k+1)-gallery filling of } A \rbrace.$$
\end{enumerate}
\end{definition}  
To make sense of the above definitions, the reader should consider the case where $k=0$ (i.e., where $X$ is a graph). In this case a $1$-gallery is just a path in the graph and for every two vertices $u,v$,
$$Fill_{1} (\lbrace u,v \rbrace) = d(u,v).$$
Using the above definitions we show the following:
\begin{theorem}
\label{intro distortion for LM theorem}
Let  $X \sim X_{k+1} (N,p)$, where $X_{k+1} (N,p)$ is the Linial-Meshulam random complex.  There are constants $C,K$ such that if $p = C \frac{\ln (N)}{N}$, then with high probability (i.e., the probability approaches $1$ as $N$ goes to $\infty$), for every map $g: X^{(0)} \rightarrow H$, where $H$ is a Hilbert space and $X^{(0)}$ are the vertices of $X$, the following holds:
\begin{dmath*} 
{\left( \sup_{1 \leq j_0 < ... < j_{k+1} \leq N} \dfrac{Fill_{k+1} (v_{j_0},...,v_{j_{k+1}})}{Vol_{k+1} (conv (g(v_{j_0}),...,g(v_{j_{k+1}})))} \right) \cdot } \\
{ \left( \sup_{1 \leq j_0 < ... < j_{k+1} \leq N} \dfrac{Vol_{k+1} (conv (g(v_{j_0}),...,g(v_{j_{k+1}}))}{Fill_{k+1} (v_{j_0},...,v_{j_{k+1}}))} \right)} \geq K \dfrac{\ln (N)}{\ln ( \ln (N))},
\end{dmath*}
where $Vol_{k+1} (conv (g(v_{j_0}),...,g(v_{j_{k+1}})))$ denotes the $(k+1)$-volume of the convex hull of $g(v_{j_0}),...,g(v_{j_{k+1}})$. 
\end{theorem}

The reader should note that when $k=0$, $X_{1} (N,p)$ is the Erd\H{o}s-R\'enyi random graph and we reproduce the theorem of Bourgain stated above. 

\begin{remark}
We note that theorem \ref{intro distortion for LM theorem} is a "watered-down" version of the more general (but harder to state) theorem \ref{general distortion for LM} proven below. In fact, this paper tries to quantizes a distortion phenomenon which does not occur for graphs - we try to measure the distortion of $k$-spheres in $X$ when they are mapped smoothly (but not necessarily affinely) into $H$. The definition of distortion becomes non-trivial, because one should consider how much a sphere is twisted by the map. Hopefully, this remark will become clearer to the reader after reading section 2 below. 
\end{remark}

\textbf{The structure of this paper.} Section 2 introduces our definition of higher dimensional distortion. We tried our best to break it down to its components in order to make it easy for the reader. Section 3 and section 4 both contain background material needed for the proof of our main technical result: section 3 introduces high dimensional Laplacians and section 4 deals with polytope boundaries and Stokes' theorem. Section 5 contains our main technical result which is a method to give a lower bound on the distortion. In section 6, we apply our technical result to obtain a lower bound on the distortion of  Linial-Meshulam random complexes.

\section{High dimensional distortion - suggested definition}
Let us start by describing a general scheme for the definition of $k$-distortion. Basically, the idea is to compare two notions of filling and compare them. We'll start with the following definition:
\begin{definition}
An abstract $k$-dimensional simplicial complex $S$ will be called a $k$-dimensional polytope boundary, if there is a convex $(k+1)$-dimensional polytope $P_S$ in $\mathbb{R}^{k+1}$ such that the boundary of $P_S$ (when considered as an abstract simplicial complex) is isomorphic to $S$.
\end{definition}
Next, we'll compare the filling of $k$-dimensional polytope boundaries:
\begin{enumerate}
\item Let $X$ be a simplicial complex $X$ of dimension $n$\footnote{$X$ should fulfil some conditions to be specified later.} , let $0 \leq k <n$ and let $S$ be a $k$-dimensional polytope boundary in $X$. The $(k+1)$-filling of $S$ inside $X$, denoted $Fill_{k+1} (S)$ should be defined such that:
\begin{itemize}
\item $0 \leq Fill_k (S)$.
\item In the $k=0$, where $S = \lbrace v_0, v_1 \rbrace$ (where $v_0, v_1$ are vertices of $X$) the definition of  $Fill_0 (S)$ should be the distance in the $1$-skeleton of $X$ between $v_0$ and $v_1$. 
\end{itemize}
\item Let $0 \leq k$ be an integer, $H$ be a Hilbert space. First, one should define a class of admissible maps $f: S \rightarrow H$ where $S$ is a $k$-dimensional polytope boundary. Next, for a $k$-dimensional polytope boundary $S$ and an admissible map $f : S \rightarrow H$, the $(k+1)$-filling of $f(S)$ in $H$, denoted $Fill_{k+1,H}  (f(S))$, should be defined such that:
\begin{itemize}
\item $0 \leq Fill_{k+1,H}  (f(S))$.
\item In the $k=0$, where $S = \lbrace x_0, x_1 \rbrace$ and $f(S) = \lbrace f(x_0), f(x_1) \rbrace$ is a set of two points in $H$, the definition of  $Fill_{0, H} (f(S))$ should be the Euclidean distance between $f(x_0)$ and $f(x_1)$.
\end{itemize}
\item Let $X$ be a simplicial complex of dimension $n$, $0 \leq k < n$ and $H$ be a Hilbert space. Fix a set $\mathcal{S}$ of $k$-dimensional polytope boundaries in $X$.  A map $f: X \rightarrow H$ will be call admissible with respect to $\mathcal{S}$, if the restriction of $f$ on every $k$-dimensional simplicial sphere $S \in \mathcal{S}$ is admissible. For a map $f: X \rightarrow H$ that is admissible with respect to $\mathcal{S}$, define the $k$-distortion of $f$ with respect to $\mathcal{S}$ as
$$distor_{k, \mathcal{S}} (f) = \sup_{S \in \mathcal{S}} \dfrac{Fill_{k+1,H}  (f(S))}{Fill_{k+1} (S)}  \sup_{S \in \mathcal{S}} \dfrac{Fill_{k+1} (S) }{Fill_{k+1,H}  (f(S))}.$$
\end{enumerate}   

The scheme for defining $distor_{k, \mathcal{S}} (f)$ stated above is of course very broad, since one can choose $Fill_{k+1}, Fill_{k+1,H}$ to be almost anything when $k \geq 1$ (most the choices wouldn't be interesting). We specified the above scheme to give the reader a sense of where we are going, before going in to the specific definitions, which can be a little technical.

\subsection{Filling in a simplicial complex}

\begin{definition}
Let $X$ be a simplicial complex of dimension $n$. 
\begin{enumerate}
\item For $-1 \leq k \leq n$, denote by $X^{(k)}$ the set of all $k$-dimensional simplices in $X$. Also, denote by $\Sigma (k)$ the set of all ordered $k$ simplices.
\item For $1 \leq k \leq n-1$, a $(k+1)$-gallery in $X$ is a sequence of $(k+1)$-dimensional simplices $\sigma_1,...,\sigma_l$ such that for every $1 \leq i \leq l-1$, $\sigma_i \cap \sigma_{i+1}$ is a simplex of dimension $k$.
\item Given $\eta_0, \eta_1 \in X^{(k)}$, we shall say that a $(k+1)$-gallery $\sigma_1,...,\sigma_l$ connects $\eta_0$ and $\eta_1$ if  $\eta_0 \subset \sigma_0, \eta_1 \subset \sigma_l$, i.e., $\eta_0$ is a face of $\sigma_0$ and $\eta_1$ is a face of $\sigma_l$. $X$ will be called $(k+1)$-gallery connected, if for every two simplices $\eta_0,\eta_1 \in X^{(k)}$ are connected by a $(k+1)$-gallery.
\item For $\eta_0,\eta_1 \in X^{(k)}$, we will say that a set $F \subseteq X^{(k+1)}$ $(k+1)$-gallery connects $\eta_0$ and $\eta_1$ if there is a $(k+1)$-gallery, $\sigma_1 \in F,...,\sigma_l \in F$,  that connects $\eta_0$ and $\eta_1$. 
\item Define the $(k+1)$-gallery distance on $X^{(k)}$ as 
$$d_{k+1} (\eta_0, \eta_1) = \min \lbrace l  : \exists \text{a (k+1)-gallery }  \sigma_1,...,\sigma_l \text{ that connects } \eta_0 \text{ and } \eta_1 \rbrace .$$
\item For any set $S \subseteq X^{(k)}$, we will say that a set $F$ is $(k+1)$-gallery filling of $S$ if for every $\eta_0, \eta_1 \in S$ we have that  $F \subseteq X^{(k+1)}$ $(k+1)$-gallery connects $\eta_0$ and $\eta_1$. Define the $(k+1)$-gallery filling number of $S$ as
$$Fill_{k+1} (S) = \min \lbrace \vert F \vert : F \subseteq X^{(k+1)}, F \text{ is (k+1)-gallery filling of } A \rbrace.$$
As mentioned above, we shall only consider the case where $S$ is a $k$-dimensional polytope boundary.
\end{enumerate}
\end{definition}

\begin{remark}
Notice that when $k=0$ and $S = \lbrace v_0, v_1 \rbrace$ we have that $Fill_{0} (S)$ is the path distance in the $1$-skeleton of $X$.
\end{remark}

\subsection{Projection volume and filling in a Hilbert space}

Let $\Omega^l$ be a $l$-dimensional compact oriented manifold with piecewise smooth boundary (see remark below) embedded in $\mathbb{R}^m$ (obviously, $l \leq m$). Fix a coordinate system $(x_1,...,x_m)$ of $\mathbb{R}^l$ and define the projection volume of $\Omega^l$ as:
$$vol_{proj} (\Omega^l) = \sqrt{\sum_{1 \leq i_1 < i_2 <...< i_l \leq m} \left( \int_{\Omega^l} dx_{i_1} \wedge dx_{i_2} \wedge ... \wedge dx_{i_l} \right)^2 }.$$

\begin{remark}
The definition of a manifold with piecewise smooth boundary can be found in \cite{AMR}[Definition 7.2.18]. The key fact about manifold with piecewise smooth boundary is that Stokes' theorem holds (see \cite{AMR}[Theorem 7.2.20]).
\end{remark}

\begin{proposition}
\label{projection volume ineq - prop}
For $\Omega^l$ as above, let $vol (\Omega^l)$ denote the $l$-dimensional volume of $\Omega_l$, then
$$vol (\Omega^l) \geq vol_{proj} (\Omega_l)  .$$
\end{proposition}

\begin{proof}
By definition we have that
\begin{dmath*}
\left( vol (\Omega^l) \right)^2 = \left( \int_{\Omega^l} \sqrt{\sum_{1 \leq i_1 <...< i_l \leq m} \left( dx_{i_1} \wedge ... \wedge dx_{i_l} \right)^2} \right)^2 = \int_{\Omega^l} \sqrt{\sum_{1 \leq i_1 <...< i_l \leq m} \left( dx_{i_1}  \wedge ... \wedge dx_{i_l} \right)^2}  \int_{\Omega^l} \sqrt{\sum_{1 \leq i_1 < ...< i_l \leq m} \left( dy_{i_1} \wedge ... \wedge dy_{i_l} \right)^2}  \geq^{(CS)} \int_{\Omega^l} \int_{\Omega^l} \left( \sum_{1 \leq i_1 <...< i_l \leq m}  dx_{i_1} \wedge ... \wedge dx_{i_l} \right) \left( \sum_{1 \leq i_1 <...< i_l \leq m} dy_{i_1} \wedge ... \wedge dy_{i_l}  \right) = \left(vol_{proj} (\Omega_l) \right)^2
\end{dmath*}
\end{proof}

Next, we'll note that $vol_{proj} (\Omega^l)$ actually depends only on the boundary of $\Omega^l$:
\begin{proposition}
Let $\Omega^l$ as above, then 
$$( vol_{proj} (\Omega^l) )^2 = \sum_{1 \leq i_1 < i_2 <...< i_l \leq m} \left( \int_{\partial \Omega^l} x_{i_1}  dx_{i_2} \wedge ... \wedge dx_{i_l} \right)^2  .$$
\end{proposition}

\begin{proof}
As mentioned above Stokes' theorem holds for manifolds with piecewise smooth boundary according to \cite{AMR}[Theorem 7.2.20]. Therefore, for any choice of $1 \leq i_1 < i_2 <...< i_l \leq m$ we have that
$$\int_{\partial \Omega^l} x_{i_1}  dx_{i_2} \wedge ... \wedge dx_{i_l} = \int_{\Omega^l} dx_{i_1} \wedge  dx_{i_2} \wedge ... \wedge dx_{i_l},$$
and the proposition follows.
\end{proof}

This give raise to the following definitions:
\begin{definition}
A compact set $B \subset \mathbb{R}^m$ will be called a $(l-1)$-dimensional piecewise smooth boundary, if it is a boundary of a $l$-dimensional manifold with piecewise smooth boundary. For $B \subset \mathbb{R}^m$ which is a $(l-1)$-dimensional piecewise smooth boundary define the following quantities:
\begin{enumerate}
\item Enclosed projection volume of $B$ defined as 
$$EnVol_{proj} (B) = \sqrt{ \sum_{1 \leq i_1 < i_2 <...< i_l \leq m} \left( \int_{B} x_{i_1}  dx_{i_2} \wedge ... \wedge dx_{i_l} \right)^2 } = vol_{proj} (\Omega) ,$$
where $\Omega$ is any smooth manifold such that $\partial \Omega = B$.
\item Enclosed volume of $B$ defined as
$$EnVol (B) = \inf_{\partial \Omega = B}  vol (\Omega) ,$$
where the infimum is taken over all $l$-dimensional compact oriented manifold with piecewise smooth boundary.
\end{enumerate}
\end{definition}

\begin{remark}
\label{flatness remark}
Note that by proposition \ref{projection volume ineq - prop}, we always have that
$$EnVol (B) \geq EnVol_{proj} (B).$$
Also note that if there is $\Omega^l$ with $\partial \Omega^l = B$ and an affine $l$-dimensional subspace, $A$ such that $\Omega^l \subset A$, then $EnVol (B) = Vol (\Omega^l) = EnVol_{proj} (B)$. In particular, if $B = \lbrace x_0, x_1 \rbrace \subset H$ is a $0$-dimensional boundary, then there is always a line segment connecting $x_0$ and $x_1$ and  $EnVol (B) = dist (x_0, x_1)$, where $dist$ is the usual Euclidean distance in $H$.   \\
We should also remark that there are other examples of boundaries $B$ where $EnVol (B) = EnVol_{proj} (B)$ that does not satisfy the condition above. Consider for instance, $B$ to be the boundary of two (full) triangles glued along an edge with dihedral angle of $\frac{\pi}{2}$.One can easily verify that $EnVol_{proj} (B)$ is just the sum of the area of the two triangles. While the gluing of the two full triangle is not a smooth manifold, it can be approached by a sequence of smooth manifolds and therefore $EnVol (B) = EnVol_{proj} (B)$.
\end{remark}
Next, will use $EnVol_{proj} (B)$ as our notion of filling in a Hilbert space:

\begin{definition}
\label{admissible on S}
Let $S$ be a $k$-dimensional polytope boundary and $H$ be a Hilbert space. Fix $P_S \subset \mathbb{R}^{k+1}$ and an identification between the boundary of $P_S$ and $S$. A map $f: S \rightarrow H$ will be called admissible if the is an open neighbourhood $U \subset \mathbb{R}^{k+1}$ of $P_S$ such that $f$ can be extended to $\widetilde{f} : U \rightarrow H$ such that:
\begin{enumerate}
\item There is a subspace $\mathbb{R}^m$ such that $\widetilde{f} (U) \subset \mathbb{R}^m$.
\item $\widetilde{f} : U \rightarrow \mathbb{R}^m$ is a smooth map. 
\end{enumerate}
Note that by above definition, $f(P_S)$ is a manifold with piecewise smooth boundary and $f(S)$ is a piecewise smooth boundary. Given an admissible map $f: S \rightarrow H$, define the $(k+1)$-filling of $f(S)$ to be $EnVol_{proj} (S)$:
$$Fill_{k+1,H} (f(S)) = EnVol_{proj} (S).$$
\end{definition}

\begin{remark}
The reader should note that $EnVol_{proj} (f(S)) $ is defined even in cases where $f$ is not admissible. In fact, all our main results hold when $EnVol_{proj} (f(S))  $ is defined and $f$ is not admissible. However, we choose to restrict ourselves to cases in which $f$ is admissible, because in those cases we can compare our definition of distortion to more intuitive terms - see discussion below.   
\end{remark}

\subsection{High dimensional distortion of a simplicial complex}

After the preceding definitions we are ready give an explicit definition of high dimensional distortion of a simplicial complex.

\begin{definition}
Let $X$ be an $n$-dimensional simplicial complex and let $\mathcal{S}$ be a set of $k$-dimensional polytope boundaries in $X$ for some $0 \leq k \leq n-1$. Assume that $X$ is $(k+1)$-gallery connected.
\begin{enumerate}
\item A map $f: X \rightarrow H$ where $H$ is a Hilbert space will be called admissible with respect to $\mathcal{S}$ (or $\mathcal{S}$-admissible) if the following holds for any choice of $\lbrace P_S : S \in \mathcal{S} \rbrace$, $f$ is admissible on $S$. To be specific for every $S \in \mathcal{S}$, there are $U_S$ an open neighbourhood of $P_S$ and an extension $\widetilde{f}_S : U \rightarrow H$ of $f$ such that the conditions of definition \ref{admissible on S} are fulfilled. As a matter of convenience we denote by $\mathbb{R}^m$ the subspace of $H$ that contains $\bigcup_{S \in \mathcal{S}} \widetilde{f}_S  (U_S)$. 
\item Define the $k$-distortion with respect to $\mathcal{S}$ of an $\mathcal{S}$-admissible map $f: X \rightarrow H$ to be: 
$$distor_{k, \mathcal{S}} (f) = \sup_{S \in \mathcal{S}} \dfrac{Fill_{k+1,H}  (f(S))}{Fill_{k+1} (S)}  \sup_{S \in \mathcal{S}} \dfrac{Fill_{k+1} (S) }{Fill_{k+1,H}  (f(S))}.$$
\end{enumerate}

\end{definition}

The above definition of distortion needs some justification. In order to convince the reader in captures some essence of distortion we shall introduce some more intuitive (and naive) definition of distortion:
\begin{definition}
Let $X$, $\mathcal{S}$ and $k$ as above. For an $\mathcal{S}$-admissible map $f: X \rightarrow H$ we define the the $k$-volume distortion of $X$ with respect to $\mathcal{S}$ as follows:
$$Vol-distor_{k, \mathcal{S}} (f) = \sup_{S \in \mathcal{S}} \dfrac{EnVol (f(S))}{Fill_{k+1} (S)}  \sup_{S \in \mathcal{S}} \dfrac{Fill_{k+1} (S) }{EnVol (f(S))}.$$
\end{definition}

The volume distortion can be easily understood as follows: assume that there is a constant $C$ such that
$$C \geq \sup_{S \in \mathcal{S}} \dfrac{EnVol (f(S))}{Fill_{k+1} (S)}.$$
Consider $\sigma \in X^{(k+1)}$, denote the combinatorial boundary of $\sigma$ as $\partial \sigma$, where here we mean that $\partial \sigma$ is just the union of all the $k$-faces of $\sigma$.  Note that $\partial \sigma$ is a $k$-dimensional polytope boundary in $X$. Assume that $\mathcal{S}$ is chosen such that $\partial \sigma \in \mathcal{S}$. Therefore, for every $\sigma \in X^{(k+1)}$, we have that 
$$C \geq \dfrac{EnVol (f(\partial \sigma))}{Fill_{k+1} (\partial \sigma)}.$$
On the other hand, by definition we have that $Fill_{k+1} (\partial \sigma) =1$, therefore for $\sigma \in X^{(k+1)}$, $C \geq EnVol (f(\partial \sigma))$. For some $S$, if our job was just to cover a $\Omega$ manifold bounded by $f(S)$ by sub-manifold of volume $C$, the number of manifolds that we would is about $\frac{Vol (\Omega)}{C}$. So if 
$$\dfrac{Fill_{k+1} (S) }{EnVol (f(S))} >>\dfrac{1}{C},$$
for some $S$ it means that we are mapping a set which is a $(k+1)$-filling of $S$ in a far from optimal way. Therefore if $Vol-distor_{k, \mathcal{S}} (f) >>1$, we should consider this a map distorting the $(k+1)$-volumes. The reader should compare the situation to a metric distortion of a graph. \\
However, volume distortion fails to capture some aspects that one may want to consider when thinking about distortion when $k >0$. Consider the following example - let $X$ be a complex that is just a single $2$-dimensional simplex (and its faces). Let $S$ be the $1$-dimensional boundary of the $2$-dimensional complex of $X$. Consider map that $f': S \rightarrow \mathbb{R}^3$ that sends $S$ to a figure $8$ in the $xy$ plane and let $f : S \rightarrow \mathbb{R}^3$ be a small perturbation of $f'$ in the $z$-coordinate such that the image $f(S)$ won't self intersect. The map $f$ can be arranged such that $EnVol (f(S))=1$ and therefore $Vol-distor_{1, \lbrace S \rbrace} (f) =1$. If one thinks of a flat image as being undistorted, then there is a distortion if $f(S)$ that isn't measured by the volume distortion. In order to measure such  distortion we wish to introduction the quotient 
$$\dfrac{EnVol (f(S))}{EnVol_{proj} (f(S))}.$$
As noted above, this quotient is always greater or equal than $1$ and intuitively it is close to $1$ when $f(S)$ is close to being flat in some way (see remark \ref{flatness remark} for examples of $\frac{EnVol (f(S))}{EnVol_{proj} (f(S))} =1$). By following the definitions, it is not hard to verify that for any $\mathcal{S}$-admissible map $f: X \rightarrow H$ one has
$$\left( Vol-distor_{k, \mathcal{S}} (f) \right) \left( \sup_{S \in \mathcal{S}} \dfrac{EnVol (f(S))}{EnVol_{proj} (f(S))}  \right) \geq distor_{k, \mathcal{S}} (f),$$
and therefore a lower bound on $distor_{k, \mathcal{S}} (f)$ gives a lower bound on 
$$\left( Vol-distor_{k, \mathcal{S}} (f) \right) \left( \sup_{S \in \mathcal{S}} \dfrac{EnVol (f(S))}{EnVol_{proj} (f(S))}  \right),$$ 
which is maybe a more intuitive quantity.

\section{High dimensional Laplacians and differentials of polytope boundaries}

This aim of this section is to provide the basic definitions regarding high dimensional Laplacians of simplicial complexes. The reader should note that in different sources, high dimension Laplacians are defined in non equivalents ways, where the difference between definitions is in using different types of normalizations (which may lead to different results in some examples).  \\ \\
Let $X$ be a pure $n$-dimensional simplcial complex, i.e, every simplex in $X$ is a face of at least one $n$-dimensional simplex. For any $0 \leq k \leq n$, we introduce the following definitions and notations:
\begin{enumerate}
\item Denote by $\Sigma (k)$ the set of ordered simplices in $X$ of dimension $k$, i.e., the set of all ordered $(i+1)$-tuples $(v_0,...,v_{k})$ such that $v_0,...,v_k \in X^{(0)}$ and $\lbrace v_0,...,v_k \rbrace \in X^{(i)}$.
\item Denote by $C^{k} (X, \mathbb{R} )$ all $k$-cocycles with values in $\mathbb{R}$, i.e., every $\phi \in C^{k} (X, \mathbb{R} )$ is a function $\phi : \Sigma (k) \rightarrow \mathbb{R}$ such that for any permutation $\pi \in Sym (\lbrace 0,...,k \rbrace)$ and any $(v_0,...,v_{k}) \in \Sigma (k)$, we have that 
$$ \phi ((v_0,...,v_{k})) = sign (\pi) \phi ((v_{\pi (0)},...,v_{\pi (k)}).$$
\item Define 
$$m : \bigcup_{k=0}^{n} X^{(k)} \rightarrow \mathbb{R}^+,$$
$$\forall 0 \leq k \leq n, \forall \tau \in X^{(k)},   m(\tau) = (n-k)! \vert \lbrace \sigma \in X^{(n)} : \tau \subseteq \sigma \rbrace \vert .$$
\item Define an inner product on $C^{k} (X, \mathbb{R} )$ as 
$$\langle \phi , \psi \rangle = \sum_{\sigma \in \Sigma (k)} \dfrac{m (\sigma)}{(k+1)!} \phi (\sigma ) \psi (\sigma ) ,$$
where $m(\sigma)$ is just $m$ applying to the simplex after forgetting the ordering.
Also, denote $\Vert \phi \Vert$ to be the norm induced by this inner product, i.e., $\Vert \phi \Vert  = \sqrt{ \langle \phi, \phi \rangle }$.
\item Define the differential $d_k : C^{k} (X, \mathbb{R} ) \rightarrow C^{k+1} (X, \mathbb{R} )$ as follows: for every $\phi \in C^k (X,\mathbb{R})$ and every $(v_0,...,v_{k+1}) \in \Sigma (k+1)$,
$$(d_k \phi ) ((v_0,...,v_{k+1} )) = \sum_{i=0}^{k+1} (-1)^i \phi ((v_0,..., \widehat{v_i},...,v_{k+1})).$$
\item One can easily check that for every $0 \leq k \leq n-1$ we have that $d_{k+1} d_k = 0$ and therefore we can define the cohomology in the usual way:
$$H^k (X, \mathbb{R} ) = \dfrac{ker (d_k)}{im (d_{k-1})} .$$
\item Let $d_{k}^* : C^{k+1} (X, \mathbb{R} ) \rightarrow  C^{k} (X, \mathbb{R} )$, be the adjoint operator of $d_k$ with respect to the inner products on $C^{k+1} (X, \mathbb{R} ),  C^{k} (X, \mathbb{R} )$. Denote $\Delta^+_k :C^{k} (X, \mathbb{R} ) \rightarrow C^{k} (X, \mathbb{R} )$ to be $\Delta^+_k = d_k^* d_k$. $\Delta^+_k$ will be called the upper $k$-Laplacian (there are also definitions of the lower $k$-Laplacian and the full $k$-Laplacian, but we won't make any use of these operators in this paper). Note that by definition, $\Delta^+_k$ is a positive operator.
\end{enumerate}

\section{Differentials of polytope boundaries and Stokes' theorem}
In this section we'll give some definitions and notations regarding polytope boundaries and recall Stokes' theorem for them. \\
Let $S$ be a simplicial complex which is $k$-dimensional polytope boundary. Identify $S$ with the boundary of $P_S \subset \mathbb{R}^{k+1}$ (note that this identification is not necessarily unique and when we write $P_S$ we actually mean the polytope and the identification). $P_S$ comes with an orientation induces from the positive orientation of $\mathbb{R}^{k+1}$ and therefore $S$ the simplices of $S$ can be oriented accordingly. Therefore, there is an orientation on all the $k$-simplices of $S$ is an oriented simplicial complex (i.e., two $k$-simplices that intersect on a $(k-1)$-face induce opposite orientations on that face). We shall denote $\Sigma_{S,+} (k)$ as the set of ordered $k$-simplices under the above orientation (each simplex of order $k$ in $S$ has only one representative in $\Sigma_{S,+} (k)$). We can define an operator $d_{P_S} : C^k (S, \mathbb{R}) \rightarrow \mathbb{R}$ as
$$d_{P_S} \phi = \sum_{\tau \in \Sigma_{S,+} (k)} \phi (\tau).$$ 
Next, we'll make the following observation:
\begin{observation}
If $S$ is  is $k$-dimensional polytope boundary inside a larger simplicial complex $X$, then $d_{P_S}$ can be defined as $d_{P_S} : C^k (X, \mathbb{R}) \rightarrow \mathbb{R}$ (in the same way).Moreover, for every $\psi \in C^{k-1} (X, \mathbb{R})$ we have that
$$d_{P_S} d_{k-1} \psi = 0.$$
\end{observation}

Let us recall Stokes' theorem in the case of $P_S$ (we omit the proof, but the interested reader can find all the details in \cite{AMR}[sections 7.2B, 7.2C]).  For a map $f: \mathbb{R}^{k+1}$ which is smooth on an open neighbourhood of $P_S$ we have an orientation on $f(P_S)$ induced by the orientation on $P_S$ and an orientation on every $f (\tau)$ for $\tau \in \Sigma_{S,+} (k)$. With these orientations, the Stokes' theorem holds: i.e., for any differential $k$-form $\omega$ we have that
$$\int_{f(P_S)} d \omega = \sum_{\tau \in \Sigma_{S,+} (k)} \int_{f(\tau)} \omega. $$
This can be rewritten in the following way - for any differential $k$-form $\omega$, define $\phi_\omega \in C^{k} (S, \mathbb{R})$ as 
$$\forall \tau \in \Sigma_{S,+} (k), \phi (\tau) = \int_{f(\tau)} \omega, $$
(and define $\phi$ on all the ordered $k$-simplices of $S$ accordingly).
Then 
$$d_{P_S} \phi_\omega = \sum_{\tau \in \Sigma_{S,+} (k)} \int_{f(\tau)} \omega = \int_{f(P_S)} d \omega,$$
or in short
$$d_{P_S} \phi_\omega = \int_{f(P_S)} d \omega.$$
Last, we shall deal with the case where $X$ is a simplicial complex, $\lbrace v_0,...,v_{k+1} \rbrace \in X^{(k+1)}$ and $S$ is the boundary of $\lbrace v_0,...,v_{k+1} \rbrace$, i.e, all simplices of the form $\lbrace v_0,..., \widehat{v_i},...,v_{k+1} \rbrace$. Let $\triangle^{k+1}$ denote the standard simplex in $\mathbb{R}^{k+1}$ spanned by $e_0= (0,...,0), e_1 = (1,0,...,0), e_2 = (0,1,0,...,0),...,e_{k+1} = (0,...,0,1)$. It is obvious that there are $(k+2)!$ ways to identify $\lbrace v_0,...,v_{k+1} \rbrace$ to $\triangle^{k+1}$. We'll follow the following convention:
choosing an order $(v_0,...,v_{k+1} ) \in \Sigma (k)$ induces the following map between $\lbrace v_0,...,v_{k+1} \rbrace$ and $\triangle^{k+1}$: $v_0$ is mapped to $e_0$,..., $v_{k+1}$ is mapped to $e_{k+1}$. Thus any ordering $(v_0,..,v_{k+1})$ defined $P_S$ which is the simplex $\triangle^{k+1}$ and the identification described above. We use the same convention to identify $(v_0,...,v_k)$ with $\triangle^{k} \subset \mathbb{R}^k$. By this convention, for $d_{(v_0,...,v_{k+1})} : C^k (X, \mathbb{R}) \rightarrow \mathbb{R}$ defined above for a general $P_S$ is simply the usual differential evaluated at $(v_0,...,v_{k+1})$:
$$d_{(v_0,...,v_{k+1})} (\phi) = d \phi ((v_0,...,v_{k+1})).$$

\section{Main technical results}

This section is devoted to proving the main technical results of this paper. Using the terminology above, the proofs become very easy. 
\begin{proposition}
Let $X$ be a pure simplicial complex of dimension $n$, and $k \leq n$. Assume that $H^k (X, \mathbb{R} ) =0$. Let $\mathcal{A}$ be a set of $k$-dimensional polytope boundaries in $X$. For every $S \in \mathcal{A}$, fix $P_S$ and define an orientation on $S$ and $d_{P_S}$ as above. Let $l , \lambda \in \mathbb{R}, s \in \mathbb{N}$ be constants such that
$$\inf_{\tau \in X^{(k)}} \dfrac{m (\tau)}{\vert \lbrace S \in \mathcal{A} : \tau \subset S \rbrace \vert} \geq l, $$
$$\sup_{S \in \mathcal{A}} \vert \lbrace \tau \in X^{(k)} : \tau \subset S \rbrace \vert \leq s.$$
$$Spec (\Delta^+_k) \setminus \lbrace 0 \rbrace \subset [ \lambda, \infty).$$
 Then for every $\phi \in C^{k} (X, \mathbb{R})$ we have that 
$$ \Vert d \phi \Vert^2 \geq \dfrac{l \lambda}{s} \sum_{S \in \mathcal{A}} ( d_{P_S} \phi )^2. $$
\end{proposition}

\begin{proof}
Fix $\phi \in C^{k} (X, \mathbb{R})$.  By the assumption that $H^k (X, \mathbb{R} ) =0$, there is $\psi \in C^{k-1} (X, \mathbb{R})$ such that 
$$(\phi + d \psi ) \perp ker (\Delta^+_k) .$$
Therefore 
\begin{dmath*}
{\Vert d \phi \Vert^2 = \langle \Delta^+_k \phi , \phi \rangle \geq \lambda \Vert \phi + d \psi  \Vert^2  =} \\
 {\lambda \sum_{\tau \in \Sigma (k)} \dfrac{m(\tau )}{(k+1)!} \left( \phi (\tau) + d \psi (\tau ) \right)^2 \geq } \\
 \lambda \sum_{S \in \mathcal{A}} \sum_{\tau \in \Sigma_{S,+} (k)} \dfrac{m(\tau )}{\vert \lbrace S \in \mathcal{A} : \tau \subset S \rbrace \vert} \geq l \lambda \sum_{S \in \mathcal{A}} \sum_{\tau \in \Sigma_{S,+} (k)} \left( \phi (\tau) + d \psi (\tau ) \right)^2,
\end{dmath*}
(Recall that $\Sigma_{S,+} (k)$ are the $k$-simplices in $S$ with the orientation induced by $P_S$). \\
Recall that for any $s$ numbers, $a_1,...,a_s$, we have that 
$$a_1^2 + a_2^2 +... +a_s^2 \geq \dfrac{1}{s} (a_1 + ... + a_s)^2.$$ 
Therefore, 
\begin{dmath*}
l \lambda \sum_{S \in \mathcal{A}} \sum_{\tau \in \Sigma_{S,+} (k)} \left( \phi (\tau) + d \psi (\tau ) \right)^2 \geq \dfrac{l \lambda}{s} \sum_{S \in \mathcal{A}} \left( d_{P_S} (\phi + d \psi ) \right)^2 = \dfrac{l \lambda}{s} \sum_{S \in \mathcal{A}} \left( d_{P_S} \phi \right)^2,
\end{dmath*}
where the last equality is due to the fact that $d_{P_S} d_{k-1} \psi = 0$.
\end{proof}

Using the above result we can prove the following:

\begin{lemma}
\label{EnVol inequality}
Let $X$ be a pure simplicial complex of dimension $n$ that is $(k+1)$-gallery connected and such that that $H^k (X, \mathbb{R} ) =0$. 
Assume that there is a set of $k$-dimensional polytope boundaries in $X$ denoted $\mathcal{A}$ and constants $l , \lambda \in \mathbb{R}, s \in \mathbb{N}$ be constants such that
$$\inf_{\tau \in X^{(k)}} \dfrac{m (\tau)}{\vert \lbrace S \in \mathcal{A} : \tau \subset S \rbrace \vert} \geq l, $$
$$\sup_{S \in \mathcal{A}} \vert \lbrace \tau \in X^{(k)} : \tau \subset S \rbrace \vert \leq s.$$
$$Spec (\Delta^+_k) \setminus \lbrace 0 \rbrace \subset [ \lambda, \infty).$$
Denote by $\mathcal{T}$ the set of boundaries of $(k+1)$-simplices in $X$. Define $\mathcal{S} = \mathcal{A} \cup \mathcal{T}$. 
Then for every map $f: X \rightarrow H$ which is admissible with respect to $\mathcal{S}$, we have that
$$ \sum_{\sigma \in X^{(k+1)}} m (\sigma ) \left( EnVol_{proj} (f(\partial \sigma )) \right)^2 \geq  \dfrac{l \lambda}{s} \sum_{S \in \mathcal{A}}  \left( EnVol_{proj} (f(S)) \right)^2 .$$

\end{lemma}

\begin{proof}
By our assumptions, there is a subspace $\mathbb{R}^m$ such that for every $S \in \mathcal{S}$, $f(S) \subset \mathbb{R}^m$. For any differential $k$-form $\omega$ we can define $\phi_\omega \in C^{k} (X, \mathbb{R} )$ by
$$\phi_\omega (\tau ) = \int_{f(\tau)} \omega.$$
Applying the above proposition to $\phi_\omega$ yields
$$ \Vert d \phi_\omega \Vert^2 \geq \dfrac{l \lambda}{s} \sum_{S \in \mathcal{A}} ( d_{P_S} \phi_\omega )^2. $$
Let $1 \leq i_0 < i_1 <...<i_{k+1} \leq m$ and define
$$\omega_{i_0,...,i_{k+1}} = x_{i_0} dx_{i_1} \wedge ... \wedge d_{i_{k+1}} .$$
Then for any such form we get
$$\sum_{\sigma \in \Sigma(k+1)} \dfrac{m (\sigma ) }{(k+2)!}\left( \int_{f (\partial \sigma)}  x_{i_0}   dx_{i_1} \wedge ... \wedge d_{i_{k+1}} \right)^2  \geq \dfrac{l \lambda}{s} \sum_{S \in \mathcal{A}}  \left( \int_{f (S)} x_{i_0}  dx_{i_1} \wedge ... \wedge d_{i_{k+1}} \right)^2.$$  
Adding all the above inequalities on all the choices of $1 \leq i_0 < i_1 <...<i_{k+1} \leq m$ yields the result stated above.
\end{proof}

With the above lemma, one can give a lower bound on the $k$-distortion:
\begin{theorem}
\label{filling distortion}
Let $X$ be a pure simplicial complex of dimension $n$ that is $(k+1)$-gallery connected and such that that $H^k (X, \mathbb{R} ) =0$. 
Assume that there is a set of $k$-dimensional polytope boundaries in $X$ denoted $\mathcal{A}$ and constants $l , \lambda \in \mathbb{R}, s \in \mathbb{N}$ be constants such that
$$\inf_{\tau \in X^{(k)}} \dfrac{m (\tau)}{\vert \lbrace S \in \mathcal{A} : \tau \subset S \rbrace \vert} \geq l, $$
$$\sup_{S \in \mathcal{A}} \vert \lbrace \tau \in X^{(k)} : \tau \subset S \rbrace \vert \leq s.$$
$$Spec (\Delta^+_k) \setminus \lbrace 0 \rbrace \subset [ \lambda, \infty).$$
Denote
$$D = \max_{\tau \in X^{(k+1)}} \vert \lbrace \sigma \in X^{(k+1)} : \tau \subset \sigma \rbrace \vert. $$  
Also denote $\mathcal{T}$ as the set of boundaries of $(k+1)$-simplices in $X$ and $\mathcal{S} = \mathcal{A} \cup \mathcal{T}$. 
Then for every map $f: X \rightarrow H$ which is admissible with respect to $\mathcal{S}$, we have that
$$distor_{k, \mathcal{S}} (f) \geq \left( \dfrac{\ln (\frac{\lfloor \frac{\vert \mathcal{A} \vert}{2} \rfloor}{\vert X^{(k)} \vert}) - s \ln (2)}{(s-1) \ln (Dk)} -1 \right) \left(\sqrt{ \frac{(n+1)! s}{2 (k+2)! l \lambda} \frac{\vert X^{(n)} \vert}{\vert \mathcal{A} \vert} } \right)^{-1} .$$
\end{theorem}

Before proving the above theorem we'll need an additional combinatorial claim:
\begin{claim}
\label{combinatorial filling bound claim}
Let $X$ be a pure simplicial complex of dimension $n$ that is $(k+1)$-gallery connected and let $\mathcal{B}$ a set of $k$-dimensional polytope boundaries in $X$. Let $s \in \mathbb{N}$ be a constant such that 
$$\sup_{S \in \mathcal{B}} \vert \lbrace \tau \in X^{(k)} : \tau \subset S \rbrace \vert \leq s,$$
and denote 
$$D = \max_{\tau \in X^{(k+1)}} \vert \lbrace \sigma \in X^{(k+1)} : \tau \subset \sigma \rbrace \vert. $$
Then there is $S \in \mathcal{B}$ with
$$Fill_{k+1} (S) \geq   \dfrac{\ln (\frac{\vert \mathcal{B} \vert}{\vert X^{(k)} \vert}) - s \ln (2)}{(s-1) \ln (Dk)} -1  .$$
\end{claim}

\begin{proof}
Let $A_k (s,r)$ be defined as follows
$$A_k (s,r) = \lbrace S \subset X^{(k)} : \vert S \vert \leq s, Fill_{k+1} (S) \leq r \rbrace ,$$
i.e.,  $A_k (s,r)$ is the set of all sets with at most $s$ elements in $X^{(k)}$ and $(k+1)$-gallery filling less or equal to $r$. Note that in order to prove the claim, it is enough to prove that for every 
$$r < \dfrac{\ln (\frac{\vert \mathcal{B} \vert}{\vert X^{(k)} \vert}) - s \ln (2)}{(s-1) \ln (Dk)} -1 ,$$
we have that
$$\left\vert A_k \left( s, r \right) \right\vert < \vert \mathcal{B} \vert.$$
Let $S \in A_k (s,r)$ and $\tau \in S$. By the definition of the $(k+1)$-gallery filling, we have that for every $\tau' \in S$ that $d_{k+1} (\tau, \tau' ) \leq r$. In other words, if we denote the $r$ (closed) ball around $\tau$ (with respect to  $d_{k+1}$) as:
$$B(\tau, r) = \lbrace \tau' \in X^{(k)} : d_{k+1} (\tau, \tau' ) \leq r \rbrace,$$
we have that 
$$\forall \tau' \in S, \tau' \in B(\tau, r) .$$
By 
$$D = \max_{\tau \in X^{(k+1)}} \vert \lbrace \sigma \in X^{(k+1)} : \tau \subset \sigma \rbrace \vert. $$
we have that for every $\eta \in X^{(k)}$, we have that 
$$\vert \eta' \in X^{(k)} : d_{k+1} (\eta , \eta') =1 \rbrace \vert \leq Dk,$$
and therefore
$$\vert B(\tau, r) \vert \leq 1 + Dk + (Dk)^2 + ... + (Dk)^r \leq (Dk)^{r+1} .$$
This yields that for a fixed $\tau \in X^{(k)}$ there are at most $2^s (Dk)^{(s-1)(r+1)}$ sets in $A(k,r)$ that contain $\tau$. Therefore,
$$\vert  A_k (s,r) \vert \leq   \vert X^{(k)} \vert 2^s (Dk)^{(s-1)(r+1)} .$$
so for every $r$ with
$$r < \dfrac{\ln (\frac{\vert \mathcal{B} \vert}{\vert X^{(k)} \vert}) - s \ln (2)}{(s-1) \ln (Dk)} -1 ,$$
we get that 
$$\left\vert A_k \left( s, r \right) \right\vert < \vert \mathcal{B} \vert ,$$
which finishes the proof.

\end{proof}

\begin{remark}
The reader should note that the estimates in claim \ref{combinatorial filling bound claim} are very rough and improving them might also improve the bound on the distortion in theorem \ref{filling distortion}. 
\end{remark}

Next, we'll prove theorem \ref{filling distortion}:
\begin{proof}
Let $f: X \rightarrow H$ be a map which is admissible with respect to $\mathcal{S}$ defined above. Denote 
$$C= \sup_{S \in \mathcal{S}} \dfrac{Fill_{k+1,H}  (f(S))}{Fill_{k+1} (S)} .$$
Note that for every $\sigma \in X^{(k+1)}$, we have the that $Fill_{k+1} (\partial \sigma) = 1$.  Therefore, every $\sigma \in X^{(k+1)}$, 
$$ C \geq Fill_{k+1,H}  (f(\partial \sigma)) = EnVol_{proj} (f(\partial \sigma ) .$$
By the above lemma we get that 
\begin{dmath*}
\sum_{\sigma \in X^{(k+1)}} m (\sigma ) C^2 \geq \sum_{\sigma \in X^{(k+1)}} m (\sigma ) \left( EnVol_{proj} (f(\partial \sigma )) \right)^2 \geq  \dfrac{l \lambda}{s} \sum_{S \in \mathcal{A}}  \left( EnVol_{proj} (f(S)) \right)^2 .
\end{dmath*}
Note that by definition, 
$$\sum_{\sigma \in X^{(k+1)}} m (\sigma ) = \sum_{\sigma \in X^{(k+1)}} (n-k-1)! \vert \lbrace \eta \in X^{(n)} : \sigma \subseteq \eta \rbrace \vert = \dfrac{(n+1)!}{(k+2)!} \vert X^{(n)} \vert.$$
Therefore we have that 
$$C^2 \dfrac{(n+1)!}{(k+2)!} \vert X^{(n)} \vert \geq \dfrac{l \lambda}{s} \sum_{S \in \mathcal{A}}  \left( EnVol_{proj} (f(S)) \right)^2 ,$$
which yields
$$C^2 \dfrac{(n+1)! s}{(k+2)! l \lambda} \vert X^{(n)} \vert \geq \sum_{S \in \mathcal{A}}  \left( EnVol_{proj} (f(S)) \right)^2 .$$
Note that by this inequality, the median of the (multi)set $\lbrace \left( EnVol_{proj} (f(S)) \right)^2 : S \in \mathcal{A} \rbrace$ is less or equal to $\frac{1}{2} C^2 \frac{(n+1)! s}{(k+2)! l \lambda} \frac{\vert X^{(n)} \vert}{\vert \mathcal{A} \vert}$. Therefore there are at least $\lfloor \frac{\vert \mathcal{A} \vert}{2} \rfloor$ elements $S \in \mathcal{A}$ such that
$$\left( EnVol_{proj} (f(S)) \right)^2 \leq \frac{1}{2} C^2 \frac{(n+1)! s}{(k+2)! l \lambda} \frac{\vert X^{(n)} \vert}{\vert \mathcal{A} \vert} .$$
Denote 
$$\mathcal{B} = \left\lbrace S \in \mathcal{A} : \left( EnVol_{proj} (f(S)) \right)^2 \leq \frac{1}{2} C^2 \frac{(n+1)! s}{(k+2)! l \lambda} \frac{\vert X^{(n)} \vert}{\vert \mathcal{A} \vert}  \right\rbrace.$$
Then $\vert \mathcal{B} \vert \geq \lfloor \frac{\vert \mathcal{A} \vert}{2} \rfloor$ and therefore by the above claim, there is $S_0 \in \mathcal{B}$ such that 
$$Fill_{k+1} (S_0) \geq   \dfrac{\ln (\frac{\lfloor \frac{\vert \mathcal{A} \vert}{2} \rfloor}{\vert X^{(k)} \vert}) - s \ln (2)}{(s-1) \ln (Dk)} -1  .$$
From the fact that $S_0 \in \mathcal{B}$, we also have 
$$ EnVol_{proj} (f(S_0)) \leq  C \sqrt{ \frac{(n+1)! s}{2 (k+2)! l \lambda} \frac{\vert X^{(n)} \vert}{\vert \mathcal{A} \vert} }.$$
Therefore
$$\dfrac{Fill_{k+1} (S_0) }{Fill_{k+1,H}  (f(S_0))} \geq \left( \dfrac{\ln (\frac{ \lfloor \frac{\vert \mathcal{A} \vert}{2} \rfloor}{ \vert X^{(k)} \vert}) - s \ln (2)}{(s-1) \ln (Dk)} -1 \right) \left( C \sqrt{ \frac{(n+1)! s}{2 (k+2)! l \lambda} \frac{\vert X^{(n)} \vert}{\vert \mathcal{A} \vert} } \right)^{-1} .$$
This in turn yields that 
$$distor_{k, \mathcal{S}} (f) \geq \left( \dfrac{\ln (\frac{ \lfloor \frac{\vert \mathcal{A} \vert}{2} \rfloor}{ \vert X^{(k)} \vert})- s \ln (2)}{(s-1) \ln (Dk)} -1 \right) \left(\sqrt{ \frac{(n+1)! s}{2 (k+2)! l \lambda} \frac{\vert X^{(n)} \vert}{\vert \mathcal{A} \vert} } \right)^{-1} .$$

\end{proof}

It is important to note that the expression 
$$\left(\sqrt{ \frac{(n+1)! s}{2 (k+2)! l \lambda} \frac{\vert X^{(n)} \vert}{\vert \mathcal{A} \vert} } \right)^{-1}$$
has only a very limited contribution to the lower bound of the distortion:
\begin{claim}
For any $X$, $\mathcal{A}$, $s,l$ as above, we have that 
$$\left(\sqrt{ \frac{(n+1)! s}{2 (k+2)! l \lambda} \frac{\vert X^{(n)} \vert}{\vert \mathcal{A} \vert} } \right)^{-1} \leq \sqrt{2 (k+2) \lambda}.$$
\end{claim}

\begin{proof}
Recall that by definition
$$\sum_{\tau \in X^{(k)}} m (\tau ) = \sum_{\tau \in X^{(k)}} (n-k)! \vert \lbrace \eta \in X^{(n)} : \tau \subseteq \eta \rbrace \vert = \dfrac{(n+1)!}{(k+1)!} \vert X^{(n)} \vert.$$
By the choice of $l$, we have for every $\tau \in X^{(k)}$ that
$$m(\tau) \geq \vert \lbrace S \in \mathcal{A} : \tau \subset S \rbrace \vert.$$
Summing on all $\tau \in X^{(k)}$ (and recalling how $s$ was defined), we get that
$$\dfrac{(n+1)!}{(k+1)!} \vert X^{(n)} \vert \geq l \sum_{\tau \in X^{(k)}} \vert \lbrace S \in \mathcal{A} : \tau \subset S \rbrace \vert \geq \dfrac{l}{s} \vert \mathcal{A} \vert .$$
Therefore
 $$\frac{(n+1)! s}{(k+1)! l} \frac{\vert X^{(n)} \vert}{\vert \mathcal{A} \vert} \geq 1,$$
 and the claim follows.

\end{proof}

The difficulty of applying the above theorem to get a good lower bound on distortion is choosing $\mathcal{A}$. To get a large bound on distortion, $\mathcal{A}$ should be chosen such that $\vert \mathcal{A} \vert, l$ are large as possible and $s$ is as small as possible.  \\
Below, we shall use the above theorem to give a lower bound on distortion for random complexes of the Linial-Meshulam model. We conjecture that one can also use this theorem to give a lower bound for other models of random complexes.

\section{Distortion for Linial-Meshulam random complexes}

The idea behind random simplicial complexes, is to start with a set of $N$ vertices and define a simplicial complex structure on those vertices at random given some probability $p (N)$. This can be done in several ways, called different models of random complexes. We shall only consider the Linial-Meshulam model (see \cite{LM}, \cite{MW} for further discussion and definitions regarding this model) that can be described as follows: for $0 \leq k$, $X \sim X_{k+1} (N,p)$ ($X$ which is distributed according to $X_{k+1} (N,p)$) is a random $(k+1)$-dimensional simplicial complex on $N$ vertices defined as follows: 
\begin{enumerate}
\item Denote the vertices of a $X$ by 
$$X^{(0)} = \lbrace v_1,...,v_N \rbrace$$
\item $X$ has a complete $k$-skeleton, i.e., for every $1 \leq j_0 < ... < j_k \leq N$, $\lbrace v_{j_0},...,v_{j_k} \rbrace$ is a $k$-simplex.
\item $(k+1)$-simplices of $X$ are chosen at random according to the following rule: for every $1 \leq j_0 < ... < j_{k+1} \leq N$,  $\lbrace v_{j_0},...,v_{j_k} \rbrace$ is a $(k+1)$-simplex.
\end{enumerate}
The reader should note that, $X \sim X_{1} (N,p)$ is the Erd\H{o}s-R\'enyi random graph. \\
In general, we concern ourselves with the asymptotic behaviour of random complexes as $N$ goes to infinity. This is done by introducing the notion of properties that happen with high probability: we say that a property $P$ happens with high probability if
$$\lim_{N \rightarrow \infty} \mathbb{P} (\lbrace X \sim X(N,p) \text{ has } P \rbrace ) =1.$$

Next, state the following results from \cite{GW}:
\begin{theorem}\cite{GW}[Theorem 13]
For all $c >0$ and $k \geq 0$, there are constants $C(k,c) >0$ and $c' >0$ such that for every $0.99 \geq p \geq C \frac{\ln (N)}{N}$, $X \sim X_{k+1} (N,p)$ has the following properties with probability $\geq 1 - n^c$:
\begin{enumerate}
\item $H^k (X, \mathbb{R} ) = 0$.
\item $Spec (\Delta^+_k) \setminus \lbrace 0 \rbrace \subseteq [1 - \frac{c'}{\sqrt{p(N-k-1)}}, 1 + \frac{c'}{\sqrt{p(N-k-1)}}]$.
\item $X$ is $(k+1)$-pure (this appears in the proof of \cite{GW}[Theorem 13]). 
\item $X$ is $(k+1)$-gallery connected (this can be inferred form the proof of \cite{GW}[Theorem 13] - see appendix below).
\end{enumerate}

\end{theorem}

In order to apply theorem \ref{filling distortion} for $X \sim X^{k+1} (N,p)$, we'll denote for $1 \leq j_0 < ... < j_{k+1} \leq N$, 
$$S_{j_0,...,j_{k+1}} = \lbrace \lbrace v_{j_0},..., \widehat{v_{j_i}},...,v_{j_{k+1}} \rbrace : 0 \leq i \leq k+1 \rbrace.$$
Notice that $S_{j_0,...,j_{k+1}}$ is always a polytope boundary (where the polytope is a $(k+1)$-simplex). Take
$$\mathcal{A}= \lbrace S_{j_0,...,j_{k+1}} :  1 \leq j_0 < ... < j_{k+1} \leq N \rbrace .$$
In the notations of theorem \ref{filling distortion}, we have that $s = k+2, \vert \mathcal{A} \vert = {N \choose k+2}, \vert X^{(k)} \vert = {N \choose k+1}$ (note that $\mathcal{S} = \mathcal{A}$). Next, we'll estimate $l, \vert X^{(k+1)} \vert, D$ and apply theorem \ref{filling distortion}:
\begin{proposition}
Let $X \sim X_{k+1} (N,p)$ with $0.99 \geq p \geq C \frac{\ln (N)}{N}$ (where $C>0$ is some constant). Fix $0 < \varepsilon < 1$, then with high probability:
\begin{enumerate}
\item $\vert X^{(k+1)} \vert \leq p {N \choose k+2} (1+ \varepsilon) $.
\item $D \leq p (N-k-1) (1+ \varepsilon)$.
\item $l \geq p (1- \varepsilon)$.
\end{enumerate}
\end{proposition}

\begin{proof}
\begin{enumerate}
\item $\vert X^{(k+1)} \vert$ follows the binomial distribution $B({N \choose k+2}, p)$. Therefore it has a mean $p {N \choose k+2}$ and by Chernoff Bound
$$\mathbb{P} (\lbrace \vert X^{(k+1)} \vert > p {N \choose k+2} (1+ \varepsilon) \rbrace) \leq e^{\frac{-\varepsilon^2 p {N \choose k+2}}{2 + \varepsilon}},$$
and the right hand side of this inequality goes to $1$ as $N$ goes to $\infty$.
\item Since the dimension of $X$ is $(k+1)$, estimating $D$ is estimating $m(\tau)$ for every $\tau$ that is a $k$-dimensional simplex. Fix $\tau$, then $m(\tau)$ follows the binomial distribution $B(N-k-1, p)$, therefore by Chernoff Bound, we have that 
$$\mathbb{P} (\lbrace m (\tau) > p (N-k-1) (1+ \varepsilon) \rbrace) \leq e^{\frac{-\varepsilon^2 p (N-k-1)}{2 + \varepsilon}}.$$
By union bound,
$$\mathbb{P} (\lbrace \exists \tau \in X^{(k)}, m(\tau) > p (N-k-1) (1+ \varepsilon) \rbrace) \leq {N \choose k+1} e^{\frac{-\varepsilon^2 p (N-k-1)}{2 + \varepsilon}}.$$
and the right hand side of this inequality goes to $1$ as $N$ goes to $\infty$.
\item Note that for every $\tau \in X^{(k)}$, 
$$ \vert \lbrace S \in \mathcal{A} : \tau \subset S \rbrace \vert = N-k-1 .$$
Therefore, we need to show that in high probability we have for every $\tau$ that 
$$m(\tau) \geq   p (N-k-1) (1- \varepsilon) .$$
This is done by just repeating the argument for bounding $D$ to get a lower bound, therefore we'll let the reader complete the details.
\end{enumerate}
\end{proof}

Next, we can apply theorem \ref{filling distortion} and get the following:
\begin{theorem}
\label{general distortion for LM}
Let $c >0$ fixed and let $C=C(k,c)$ be the constant from \cite{GW}[Theorem 13] stated above. Let  $X \sim X_{k+1} (N,p)$ with $0.99 \geq p \geq C \frac{\ln (N)}{N}$ and $\mathcal{A}$ as above. Then with high probability we have that for every $\mathcal{A}$-admissible map $f: X \rightarrow H$, there is a constant $K'=K' (k)$ such that with high probability
$$distor_{k, \mathcal{A}} (f) \geq K' \dfrac{\ln (N)}{\ln (pN)} .$$
In particular, for $p = C \frac{\ln (N)}{N}$ there is a constant $K = K (k)$ such that with high probability
$$distor_{k, \mathcal{A}} (f) \geq K \dfrac{\ln (N)}{\ln ( \ln (N))} .$$
\end{theorem}

\begin{proof}
By theorem \ref{filling distortion}, we have that for every  $\mathcal{A}$-admissible map $f: X \rightarrow H$ that 
$$distor_{k, \mathcal{A}} (f) \geq \left( \dfrac{\ln (\frac{ \lfloor \frac{\vert \mathcal{A} \vert}{2} \rfloor}{ \vert X^{(k)} \vert})- s \ln (2)}{(s-1) \ln (Dk)} -1 \right) \left(\sqrt{ \frac{(k+2)! s}{2 (k+2)! l \lambda} \frac{\vert X^{(k+1)} \vert}{\vert \mathcal{A} \vert} } \right)^{-1} .$$
Next, apply the above estimations for $\varepsilon = \frac{1}{2}$ and assuming that $\lambda \geq \frac{1}{2}$ (this is true with high probability by  \cite{GW}[Theorem 13] stated above). With high probability we have that
\begin{dmath*}
distor_{k, \mathcal{A}} (f) \geq \left( \dfrac{\ln (\frac{ \lfloor \frac{{N \choose k+2}}{2} \rfloor}{ {N \choose k+1}})- (k+2) \ln (2)}{(k+1) \ln (p (N-k-1) \frac{3}{2} k)} -1 \right) \left(\sqrt{ \frac{(k+2)}{p \frac{1}{2}} \frac{p {N \choose k+2} \frac{3}{2}}{{N \choose k+2}} } \right)^{-1} \geq 
\left( \dfrac{\ln (\frac{N-k-1}{2(k+2)} - \frac{1}{{N \choose k+1}})- (k+2) \ln (2)}{(k+1) \ln (p (N-k-1) \frac{3}{2} k)} -1 \right) \left(\sqrt{ \dfrac{1}{3(k+2)}} \right) \overset{N>>k}\geq
\dfrac{1}{2 (k+1)\sqrt{3(k+2)}} \dfrac{\ln (N)}{\ln (pN)} .
\end{dmath*}
For $p = C \frac{\ln (N)}{N}$, the bound stated above follows by taking for instance $K =\frac{1}{2} \frac{1}{2 (k+1)\sqrt{3(k+2)}}$.
\end{proof}

As a corollary we get theorem \ref{intro distortion for LM theorem}:

\begin{corollary}
Let $c >0$ fixed and let $C=C(k,c)$ be the constant from \cite{GW}[Theorem 13] stated above. Let  $X \sim X_{k+1} (N,p)$ with $ p = C \frac{\ln (N)}{N}$. Then there is a constant $K$ such that for any map $g: X^{(0)} \rightarrow H$, we have with high probability:
\begin{dmath*} 
{\left( \sup_{1 \leq j_0 < ... < j_{k+1} \leq N} \dfrac{Fill_{k+1} (S_{j_0,...,j_{k+1}})}{Vol_{k+1} (conv (g(v_{j_0}),...,g(v_{j_{k+1}}))} \right) \cdot } \\
{ \left( \sup_{1 \leq j_0 < ... < j_{k+1} \leq N} \dfrac{Vol_{k+1} (conv (g(v_{j_0}),...,g(v_{j_{k+1}}))}{Fill_{k+1} (S_{j_0,...,j_{k+1}})} \right)} \geq K \dfrac{\ln (N)}{\ln ( \ln (N))}.
\end{dmath*}
(If for some $1 \leq j_0 < ... < j_{k+1} \leq N$, $Vol_{k+1} (conv (g(v_{j_0}),...,g(v_{j_{k+1}})) = 0$, then the expression on the left hand side of the inequality is taken to be $\infty$).
\end{corollary}

\begin{proof}
If there are vertices $v_{j_0},...,v_{j_{k+1}} \in X^{(0)}$ such that $g(v_{j_0}),...,g(v_{j_{k+1}})$ are not in general position in $H$, we have that
$$Vol_{k+1} (conv (g(v_{j_0}),...,g(v_{j_{k+1}})) = 0,$$
and there is nothing to prove. 
Assume that $g: X^{(0)} \rightarrow H$ is such that every $k+2$ vertices are mapped to points in general position.  Define $f: X \rightarrow H$ to be the affine extension of $g$. $f$ is admissible with respect to $\mathcal{A}$ defined as in the above theorem and for every $S_{j_0,...,j_{k+1}}) \in \mathcal{A}$ we have that 
$$Fill_{k,H} (S_{j_0,...,j_{k+1}}) = Vol_{k+1} (conv (g(v_{j_0}),...,g(v_{j_{k+1}})).$$
Therefore the inequality follows by the above theorem.
\end{proof}

\appendix
\section{Gallery connectivity for simplicial complexes}
This appendix is meant to explain how to deduce that $X \sim X_{k+1} (N,p)$ is $(k+1)$-gallery connected from the proof of \cite{GW}[Theorem 13]. In order to do that, we'll prove a more general statement connecting gallery connectivity of a simplicial complex $X$ to the connectivity of the links of $X$. Recall the following definition:
\begin{definition}
Let $X$ be a simplicial complex of dimension $n$ and let $\tau = \lbrace v_0,...,v_i \rbrace \in X^{(i)}$. The link of $\tau$, denoted $X_\tau$, is a sub-complex of $X$ of dimension $n-i-1$ defined as follows:
for every $0 \leq j \leq n-i-1$, $\lbrace u_0,...,u_j \rbrace \in X_\tau^{(j)}$ if $\lbrace u_0,...,u_j \rbrace \in X^{(j)}$ and if $\lbrace v_0,...,v_i,u_0,...,u_j \rbrace \in X^{(j+i+1)}$. \\
This definition extends to $\emptyset \in X^{(-1)}$ as $X_\emptyset = X$.
\end{definition}

We'll observe the following connection between the connectivity of links a gallery connectivity:
\begin{claim}
Let $X$ be a simplicial complex of dimension $n$ and let $0 \leq k \leq n-1$. For $\eta_0, \eta_1 \in X^{(k)}$, if $\eta_0 \cap \eta_1 = \tau \in X^{(k-1)}$ and $X_\tau$ is (path) connected, then there is a $(k+1)$-gallery in $X$ connecting $\eta_0$ and $\eta_1$.
\end{claim}

\begin{proof}
Denote $\tau = \lbrace v_0,...,v_{k-1} \rbrace, \eta_0 = \lbrace v_0,...,v_{k-1}, u_0 \rbrace, \eta_1 = \lbrace v_0,...,v_{k-1}, u_1 \rbrace$. Note that $u_0,u_1 \in X_\tau^{(0)}$ and by the assumption that $X_\tau$ is connected there are $u_0 = w_0,w_1,...,w_l=u_1 \in X_\tau^{(0)}$ such that 
$$\forall 0 \leq i \leq l-1, \lbrace w_i , w_{i+1} \rbrace \in X_\tau^{(1)}.$$
To finish, we take $\sigma_i = \lbrace v_0,...,v_{k-1}, w_i, w_{i+1} \rbrace$ and get that $\sigma_0,...,\sigma_{l-1}$ is a $(k+1)$-gallery connecting $\eta_0$ and $\eta_1$.
\end{proof}
Next, we shall prove the following proposition:
\begin{proposition}
Let $X$ be a simplicial complex of dimension $n$ and let $1 \leq k \leq n-1$.  If:
\begin{enumerate}
\item Every $\eta_0', \eta_1' \in X^{(k-1)}$ are connected by a $k$-gallery.
\item For every $\tau \in X^{(k-1)}$ that $X_\tau$ is connected.
\end{enumerate}
Then every $\eta_0, \eta_1 \in X^{(k)}$ are connected by a $(k+1)$-gallery. 
\end{proposition}

\begin{proof}
Let $\eta_0, \eta_1 \in X^{(k)}$, and take $\eta_0' \subset \eta_0, \eta_1' \subset \eta_1$ such that $\eta_0', \eta_1' \in X^{(k-1)}$. By the assumptions of the proposition, there is a $k$-gallery $\gamma_0,...,\gamma_r$ connecting $\eta_0'$ and $\eta_1'$. Note that for every $0 \leq j \leq r-1$ we have that $\gamma_j \cap \gamma_{j+1} \in X^{(k-1)}$. Therefore, by the above claim (and the assumptions of the proposition) we get that for every $j$, $\gamma_j$ and $\gamma_{j+1}$ are connected by a $(k+1)$-gallery. Also, note that $\eta_0 \cap \gamma_0 = \eta_0' \in X^{(k-1)}$ and $\eta_1 \cap \gamma_r = \eta_1' \in X^{(k-1)}$, therefore by the above claim $\eta_0$ and $\gamma_0$ are connected by a $(k+1)$-gallery and $\eta_1$ and $\gamma_r$ are connected by a $(k+1)$-gallery. This yields that $\eta_0$ and $\eta_1$ are connected by a $(k+1)$-gallery and we are done.
\end{proof}

As a corollary we get the needed result:
\begin{corollary}
Let $c>0, k \geq 0$ fixed and let $C=C(k,c)$ be the constant of  \cite{GW}[Theorem 13]. Then for $0.99 \geq p \geq C N \ln (N)$ and $X \sim X_{k+1} (N,p)$, we have that every $\eta_0, \eta_1 \in X^{(k)}$ are connected by a $(k+1)$-gallery with probability $\geq 1 - n^c$.
\end{corollary}

\begin{proof}
For the fact that $X$ has a complete $k$-skeleton, it is clear that every $\eta_0', \eta_1' \in X^{(k-1)}$ are connected by a $k$-gallery. In the proof of \cite{GW}[Theorem 13], $C(k,c)$ is chosen such that for each $\tau \in X^{(k)}$, the first non trivial eigenvalue of the graph Laplacian on $X_\tau$ is bounded away from $0$ with probability $\geq 1 - n^c$. In particular, with probability $\geq 1 - n^c$, we have that for each $\tau \in X^{(k)}$, $X_\tau$ is connected and we are done by the above proposition.
\end{proof}

\bibliographystyle{alpha}
\bibliography{bibl}

\end{document}